\documentclass[12pt]{article} \textheight634pt \textwidth455pt
\usepackage{amssymb, amsmath}
\usepackage{amsthm,  amsfonts}
\usepackage{url}

\usepackage[cp1251]{inputenc}
\usepackage[all]{xy}
\usepackage{color}
\usepackage[english]{babel}

\usepackage[T2A]{fontenc}

\usepackage{euscript}

\newtheorem{theorem}{Theorem}

\newtheorem{proposition}{Proposition}

\newtheorem{remark}{Remark}

\begin{document}


\title{Equimeasurable symmetric spaces of measurable function}
\author{M.~A.~Muratov\\ V.I.Vernadsky Crimean Federal University, 
             Simferopol,  Russian Federation, \\ B.-Z.~A.~Rubstein\\ Ben Gurion University of the Negev,
             Be'er Sheva, Israel}
\maketitle

\begin{abstract}
In this paper we consider equimeasurable symmetric(rearrangement invariant) spaces $\mathbf{E}_1 = \mathbf{E}_1(\Omega_1,\mathcal{F}_1,\mu_1)$ and $\mathbf{E}_2 = \mathbf{E}_2(\Omega_2,\mathcal{F}_2,\mu_2)$ on a measure spaces $(\Omega_1, \mathcal{F}_1,\mu_1)$ and   $(\Omega_2,\mathcal{F}_2,\mu_2)$ with finite or infinite $\sigma$-finite non-atomic measures $\mu_1$ and $\mu_2$.

      If  $\mathbf{E}_1(\Omega_1,\mathcal{F}_1,\mu_1)$ be a symmetric space on a measure spaces $(\Omega_1, \mathcal{F}_1,\mu_1)$ and  $(\Omega_2,\mathcal{F}_2,\mu_2)$
  be a measure space such that $\mu_1 (\Omega_1)=\mu_2(\Omega_2)$, 
      then there exists a unique symmetric space $\mathbf{E}_2(\Omega_2,\mathcal{F}_2,\mu_2)$ on $(\Omega_2,\mathcal{F}_2,\mu_2)$, which is
  equimeasurable to $ \mathbf{E}_1(\Omega_1,\mathcal{F}_1,\mu_1)$.
  
\medskip

Keywords: Symmetric  spaces,  Banach and quasi-Banach spaces,  Lebesgue spaces, equimeasurable symmetric spaces, Stone representation.

\medskip

 2010 MSC:  46E30, 46E35, 26D10, 26D15, 46B70.
\end{abstract}

\section{Introduction}
\label{intro}

In this paper we consider symmetric (rearrangement invariant) Banach or quasi-Banach spaces $\mathbf{E}=\mathbf{E}(\Omega,\mathcal{F}, \mu)$ of measurable functions on various
 measure spaces $(\Omega,\mathcal{F}, \mu)$. See, \cite{LTz}, \cite{KrPiSe}, \cite{BeSh}, \cite{RuGrMuPa} and references therein.
      The terminology related to quasi-Banach spaces is taken from \cite{Kal}.

     The measure space, here and throughout, are assumed to have finite or infinite $\sigma$-finite and non-atomic measures $\mu$.

     Most of the works dealing with such symmetric spaces, are beginning with the following phrase:
\begin{itemize}
      \item[$\bullet$]
    Let us assume, without loss of generality, that the considered measure space $\Omega$ is the half-line $\mathbf{R}_+ = [0,\infty)$ or its finite segment
 $[0,a], \ 0<a<\infty$,  with the corresponding Lebesgue measure.
\end{itemize}

    The main aim of this our paper is {\it to give precise meaning and justify} the above assumption.

    Let two functions $f_1$ and $f_2$ on (possibly different) measure spaces $(\Omega_1,\mathcal{F}_1,\mu_1)$ and $(\Omega_2,\mathcal{F}_2,\mu_2)$
 are called {\it equimeasurable} if $\xi_{f_1,\mu_1} =\xi_{f_2,\mu_2}$, where the decreasing rearrangement $\xi_{f,\mu}$ of a
 measurable function $f$ on $(\Omega,\mathcal{F}, \mu)$ is defined by
$$
  \xi_{f,\mu}(x) = \inf \{y \in [0,+\infty) \colon \ \mu\,(\{|f| > y \}) \leq x \} \;,\; x \in [0,\infty),
$$
see,  Section  \ref{ss Symmetric spaces} below.

Two symmetric spaces $\mathbf{E}_1 = \mathbf{E}_1(\Omega_1,\mathcal{F}_1,\mu_1)$ and $\mathbf{E}_2 = \mathbf{E}_2(\Omega_2,\mathcal{F}_2,\mu_2)$ will be called {\it equimeasurable}
 if they have the same collections of decreasing rearrangements,
 $$
     \{ \xi_{f_1,\mu_1}\colon \ f_1 \in \mathbf{E}_1\} =  \{ \xi_{f_2,\mu_2}\colon \ f_1 \in \mathbf{E}_2\}\;.
 $$

    Our main result can formulated now as follows:
\begin{itemize}
           \item[$\bullet$]
      Let $\mathbf{E}_1(\Omega_1,\mathcal{F}_1,\mu_1)$ be a symmetric space on a measure spaces $(\Omega_1, \mathcal{F}_1,\mu_1)$ and let  $(\Omega_2,\mathcal{F}_2,\mu_2)$
  be a measure space such that $\mu_1 (\Omega_1)=\mu_2(\Omega_2)$.
      Then there exists a unique symmetric space $\mathbf{E}_2(\Omega_2,\mathcal{F}_2,\mu_2)$ on $(\Omega_2,\mathcal{F}_2,\mu_2)$, which is
  equimeasurable to $ \mathbf{E}_1(\Omega_1,\mathcal{F}_1,\mu_1)$.
    \end{itemize}

     We say that a measure space $(\Omega,\mathcal{F}, \mu)$ is {\it standard} if $(\Omega,\mathcal{F}, \mu) = (I, \mathcal{B}_m, m)$,
 where either $I=[0, \infty)$ or  $I=[0, a]$ with $0< a < \infty $, while $m$ is the usual Lebesgue measure on $I$ and
 $\mathcal{B}_m$ is the $m$-completion of the Borel $\sigma$-algebra $\mathcal{B} =\mathcal{B}(I)$ with respect to $m$.

     Then the above assertion may be reformulated as follows:
\begin{itemize}
 \item[$\bullet$]
      Let $ \mathbf{E}(\Omega,\mathcal{F},\mu)$ be a symmetric space on a measure space  $(\Omega,\mathcal{F},\mu)$, and $(I, \mathcal{B}_m,  m)$ be he standard measure
 space such that $\mu(\Omega)= m(I)$.
      Then there exists a unique (standard) symmetric space $\mathbf{E}(I, \mathcal{B}_m, m)$, which is equimeasurable to $ \mathbf{E}(\Omega,\mathcal{F},\mu)$.
\end{itemize}

     The central result of the paper (Theorems \ref{th main theorem 1} and \ref{th main theorem 2}) ares formulated and proved in
 Sections \ref{s Main results}, \ref{s Symmetric spaces on separable measure spaces} and
 \ref{s Symmetric spaces on general measure spaces}.

      In Section \ref{s Symmetric spaces on separable measure spaces} we consider a special case when the measure space $(\Omega,\mathcal{F}, \mu)$
 is separable.

      Namely, let \ $(\Omega, \mathcal{F},\mu)$ \ be a separable $\sigma$-finite non-atomic measure space and \
 $(I,\mathcal{B}_m, m)$ be the corresponding standard measure space.
      Then Theorem \ref{th isomorphism Phi Separable case} asserts:
\begin{itemize}
 \item[$\bullet$]
     There exists an algebraic, order, topological and integral preserving isomorphism
$$
 \Phi  \colon \mathbf{L}_0(\Omega,\mathcal{F},\mu) \to   \mathbf{L}_0(I,\mathcal{B}_m, m)
$$
 between corresponding spaces of all real measurable functions on $(\Omega,\mathcal{F},\mu)$ and $(I,\mathcal{B}_m, m)$, respectively.
 \item[$\bullet$]
    For every symmetric space $\mathbf{E}(\Omega,\mathcal{F},\mu)\subset \mathbf{L}_0(\Omega,\mathcal{F},\mu)$ the restriction $\Phi_\mathbf{E} = \Phi|_{\mathbf{E}(\Omega,\mathcal{F},\mu)}$ is an isometric
 order isomorphism between the symmetric space $\mathbf{E}(\Omega,\mathcal{F},\mu)$ and its equimeasurable standard symmetric space $\mathbf{E}(I,\mathcal{B}_m, m)$.
 \end{itemize}
     We prove the theorem in Section \ref{s Symmetric spaces on separable measure spaces}, where an explicit construction of the
 isomorphism $\Phi$ is given.
    In the case when $(\Omega,\mathcal{F},\mu)$ is a Lebesgue space, the result follows directly from Rokhlin's classification of Lebesgue space,
 see, \cite{Ro$_1$}.
    Therefore, we focus attention on separable measure spaces $(\Omega,\mathcal{F},\mu)$ which are not isomorphic to
 Lebesgue spaces, see, Sections \ref{ss Subspaces and factor-spaces}, \ref{ss The Stone representation} and
 \ref{ss Construction of  $Phi$}.

     Theorem \ref{th isomorphism Phi Separable case} implies Theorems \ref{th main theorem 1} and \ref{th main theorem 2} in the separable
 case.

      The proof of Theorems \ref{th main theorem 1} and \ref{th main theorem 2} for general (not necessarily separable) measure
 spaces $(\Omega,\mathcal{F},\mu)$ is given in Section \ref{s Symmetric spaces on general measure spaces}.

     In Section \ref{s Some consequences and applications} we deal with various properties of symmetric spaces which are invariant for
 classes of equimeasurable symmetric spaces.
     For example, minimality, maximality, reflectivity, as well as the properties (A), (B), (C) and (F) are so.
     While, in contrast to this, separability is not an invariant property of equimeasurable symmetric spaces.

\section{Main results}
\label{s Main results}

\subsection{Equimeasurable functions}
 \label{ss Equimeasurable functions}

  Let $(\Omega,\mathcal{F}, \mu)$ be a measure space with finite or infinite $\sigma$-finite non-atomic measure
 $\mu$ defined on a $\sigma$-algebra $\mathcal{F}$ of subsets in $\Omega$.
    By passing to  the $\mu$-completion $ \mathcal{F}_\mu$ we may assume that the measure space $(\Omega,\mathcal{F}, \mu)$ is
   $\mu$-complete, i.e. $\mathcal{F} = \mathcal{F}_\mu$,
$$
 A \subseteq B \in \mathcal{F}_\mu \;,\; \mu(B) = 0 \;\Longrightarrow\; A \in \mathcal{F}_\mu \;,\; \mu(A) = 0\;.
$$
We write $(I, \mathcal{B}_m, m)$ in the special case, when either
 $I=[0, \infty)$ or  $I=[0, a]$ with $0< a < \infty $, while $m$ is the usual Lebesgue measure on $I$
 and $\mathcal{B}_m$ is the $m$-completion of the Borel $\sigma$-algebra $\mathcal{B} = \mathcal{B}(I)$ with respect to $m$.

      We denote by $ \mathbf{L}_0 = \mathbf{L}_0(\Omega,\mathcal{F},\mu)$ the set of all classes $\mu$-measurable functions
  $f \colon \Omega \to \mathbf{R} $, which are equal $\mu$-almost everywhere.
  The space $\mathbf{L}_0$ is a function algebra over $\mathbf{R}$ and a conditionally complete lattice.
      We also equipped $\mathbf{L}_0$ with a metric
$$
    \delta_0(f,g) = \int_\Omega \frac{|f-g|}{1+|f-g|}\, w \,d\mu \;,\; f,g
    \in \mathbf{L}_0(\Omega, \mathcal{F}_\mu,\mu) \;\mbox{and some}\; w \in \mathbf{L}_1^+(\Omega, \mathcal{F}_\mu,\mu)\;,
$$
 which induces the topology of measure convergence on finite measure subsets of $(\Omega, \mathcal{F}, \mu)$.
    The metric space $(\mathbf{L}_0,\delta_0)$ is complete metric space with respect to the distance.

 For every $f \in \mathbf{L}_0(\Omega, \mathcal{F},\mu)$ we define the {\it (upper) distribution function}
 $\eta_{f,\mu} $ of $|f|$ by
$$
    \eta_{f,\mu}(x) := \mu\,(\{|f| > x \}) \;, \; \; \mbox{where} \;\;
    \{|f| > x \} := \{ \omega \in \Omega\colon |f(\omega)| > x \} \;.
$$

Then $\eta_{f,\mu} $ is a decreasing and right-continuous function on $ [0,+\infty)$ such that
 $ \eta_{f,\mu} (x) \in [0, \mu(\Omega]$ for all $0 \leq x < \infty$.

     For any $f \in \mathbf{L}_0(\Omega,\mathcal{F}_\mu, \mu)$ there exists a unique function $\xi_{f,\mu} $ on $[0, \infty)$ such
  that $\xi_{f.\mu}$, which is decreasing, right-continuous,
$$
   \eta_{{\xi_{f,\mu}},\,m} = \eta_{f,\mu} \ \; \mbox{and}\; \; \xi_{{\xi_{f,\mu}},\,m} =\xi_{f,\mu}\;.
$$

The function $\xi_{f,\mu} $ can be constructed as the right-continuous (generalized)
 inverse function for $\eta_{f,\,\mu}$, i.e.
$$
 \xi_{f,\mu}(x) := \inf \{y \in [0,+\infty) \colon \eta_{f,\mu}(y) \leq x \} \;,\; x \in [0,\infty) \;.
$$
     It is called the {\it decreasing rearrangement} of $|f|$ with respect to $\mu$.

     In the case $\mu(\Omega) = a <\infty$ the function $\xi_{f,\mu}$ is defined on the segment $[0, a]$ since
 $\eta_{f,\mu}(x) \leq a$ for all $x$.

In the case $\mu(\Omega) = \infty$ the function $\xi_{f,\mu}$ is defined on $[0,\infty)$ and extended to
 $[0,\infty]$ by
$$
 \xi_{f,\mu}(\infty) := \lim_{x \rightarrow\infty} \xi_{f,\mu}(x) = \inf \{ y >0 \colon \eta_{f,\mu} (y) < \infty\} \;.
$$
    It is possible in this case that $\eta_{f,\mu}(x) = \infty$ or $\xi_{f,\mu}(x) = \infty$ for some and even
 for all $x \in [0,\infty)$.
    To avoid the undesirable cases we use the subspace
$$
  \mathbf{ L}_0^\xi(\Omega,\mathcal{F},\mu) := \{ f \in\mathbf{ L}_0(\Omega,\mathcal{F},\mu) \colon \xi_{f,\mu}(x) < \infty, \ x \in I\}=
  $$
  $$
   =\{ f \in \mathbf{L}_0(\Omega,\mathcal{F},\mu) \colon \eta_{f,\mu}(\infty)= \lim_{x \to \infty}\eta_{f,\mu}(x) =0, \  x >0\}.
$$

Then by the definition, we have $\xi_{f,\mu} \in\mathbf{L}_0(I,\mathcal{B}_m,m)$ iff $f \in \mathbf{L}_0^\xi(\Omega,\mathcal{F},\mu)$
 and
$$
 \eta_{\xi_{f,\mu},m}(y) = m(\{x \geq 0 \colon \xi_{f,\mu}(x) > y \})= \xi^{-1}_{f,\,\mu}(y)= \eta_{f,\mu}(y)
  \;,\; y > 0 \;.
$$

Two non-negative functions
 $f_1 \in\mathbf{ L}_0(\Omega_1, {\mathcal{F}_\mu}_1,\mu_1)$ and $f_2 \in L_0(\Omega_2, {\mathcal{F}_\mu}_2, \mu_2)$ are called
 {\it equimeasurable} if they have the same distribution functions, $\eta_{f_1,\mu_1} =\eta_{f_2,\mu_2}$.

    The latter equality is obviously equivalent to $\xi_{f_1,\mu_1} =\xi_{f_2,\mu_2}$.

    By the definition, \ $\xi_{f,\mu} \in \mathbf{L}_0(I,\mathcal{B}_m,m)$  for  $f\in \mathbf{L}_0(\Omega,\mathcal{F},\mu)$ such that
 $\eta_{f,\mu}(\infty)~=~0$.

     It should be noted that functions
 $f_1 \in\mathbf{ L}_0(\Omega_1, \mathcal{F}_1, \mu_1)$ and $f_2 \in \mathbf{L}_0(\Omega_2,\mathcal{F}_2, \mu_2)$ are defined on  possibly
 different measure spaces $(\Omega_1,\mathcal{F}_1,\mu_1)$ and $(\Omega_2,\mathcal{F}_2,\mu_2)$, respectively, while both their
 rearrangements $\xi_{f_1,\mu_1}$ and $\xi_{f_2,\mu_2}$ are defined on the same segment
 $[0, a]$ with $a =\mu_1(\Omega_1) =\mu_1(\Omega_2)$.

\subsection{Symmetric spaces}
 \label{ss Symmetric spaces}

  A non-trivial Banach (quasi-Banach) space $(\mathbf{E},\|\cdot\|_\mathbf{E})=(\mathbf{E}(\Omega,\mathcal{F}, \mu),\|\cdot\|_{\mathbf{E}(\Omega,\mathcal{F}, \mu)})$ of real measurable
 functions on a measure space $(\Omega,\mathcal{F}, \mu)$ is called  {\it symmetric } if the following two
 conditions hold:
\begin{itemize}
 \item[1.]
  $\, f \in \mathbf{L}_0(\Omega,\mathcal{F}, \mu) \,,\, g\in \mathbf{E}$ and $|f|\leq |g|$ imply $f\in \mathbf{E}$ and $\|f\|_\mathbf{E} \leq \|g\|_\mathbf{E}$.
 \item[2.]
  $\, f \in \mathbf{L}_0(\Omega,\mathcal{F}, \mu) \,,\, g \in \mathbf{E}$ and $\eta_{f,\mu} =\eta_{g,\mu}$ imply $f\in \mathbf{E}$ and
  $\|f\|_\mathbf{E}=\| g\|_\mathbf{E}$.
\end{itemize}

  Condition 1 means that $\mathbf{E}$ is an ideal Banach (quasi-Banach) sublattice in $\mathbf{L}_0$.
  Condition 2 is the {\it symmetry (or rearrangement invariance) property} of the quasi-norm $\|\cdot\|_\mathbf{E}$.

  Thus a symmetric space is an ideal Banach (quasi-Banach) lattice with a symmetric norm (quasi-norm).

  Since $|f|\leq |g|$ implies $\eta_{f,\mu} \leq \eta_{f,\mu}$ and $\xi_{f,\mu} \leq \xi_{g,\mu}$,
  the conditions 1 and 2 can be rewritten by means of rearrangements $\xi_{f,\mu} $  as:
$$
 f \in \mathbf{L}_0(\Omega,\mathcal{F}, \mu) \,,\, g \in \mathbf{E} \;\mbox{and}\; \xi_{f,\mu} \leq \xi_{g,\mu} \;\Longrightarrow f \in \mathbf{E}
  \;\mbox{and}\; \|f\|_\mathbf{E} \leq \|g\|_\mathbf{E} \;.
$$

Turning to the classical spaces $\mathbf{L}_p(\Omega,\mathcal{F}, \mu), \ 0 < p \leq \infty$ we set
$$
 \mathbf{ L}_p(\Omega,\mathcal{F}, \mu) := \{ f \in  \mathbf{L}_0(\Omega,\mathcal{F}, \mu) \colon \|f\|_{\mathbf{L}_p(\Omega,\mathcal{F}, \mu)}  <\infty\} \;,
$$
where for $ 0 < p <\infty$,
$$
 \|f\|_{{\mathbf{L}_p}(\Omega,\mathcal{F} \mu)} = \left (\int_\Omega |f|^p d\mu \right)^{\frac{1}{p}} \;,
$$
and
$$
    \|f\|_{\mathbf{L}_\infty(\Omega,\mathcal{F}, \mu)} := \inf\{a>0  \colon \mu\{|f|>a\}=0\} ;.
$$
    Then the spaces $\mathbf{L}_p(\Omega,\mathcal{F}, \mu) $ are ideal Banach lattices for all $1\leq p \leq \infty$,
 while they are ideal quasi-Banach lattices for $0 <p < 1$.

   In the latter case, the quasi-norm $\|\cdot\|_{\mathbf{L}_p}$ is a $p$-norm, i.e.
  $$
  \|f + g \|_{\mathbf{L}_p}^p \leq \|f\|_{\mathbf{L}_p}^p + \| g \|_{\mathbf{L}_p}^p  \;,\;  f,g \in \mathbf{L}_p \;,
  $$
 whence $\mathbf{L}_p$ becomes a (complete) metric space with respect to
 the metric
 $$
 \delta_p (f,g) = \|f - g \|_{\mathbf{L}_p}^p .
 $$

 On the other hand, for every $0<p<\infty$,
$$
 \|f\|_{ \mathbf{L}_p(\Omega,\mathcal{F}, \mu)} = \left( \int_\Omega |f|^p d\mu \right)^{\frac{1}{p}}=
 \left( \int_0^\infty (\xi_{f,\mu})^p dm \right)^{\frac{1}{p}} =
 \|\xi_{f,\mu}\|_{ \mathbf{L}_p(I,\mathcal{B}_m,m)} \;,
$$
and
$$
 \|f\|_{\mathbf{L}_\infty(\Omega,\mathcal{F}, \mu)} = \max_{\,x \in \mathbf{R}^+} \xi_{f,\mu}(x)=  \xi_{f,\mu}(0)
 =\|\xi_{f,\mu}\|_{\mathbf{L}_\infty(I,\mathcal{B}_m,m)}\;,
$$
since the functions $|f| \in \mathbf{L}_0(\Omega,\mathcal{F}, \mu)$ and $\xi_{f,\mu}
\in \mathbf{L}_0(I,\mathcal{B}_m,m)$ are equimeasurable.

   These equalities imply that $\mathbf{L}_p(\Omega,\mathcal{F}, \mu)$ is a symmetric (quasi-Banach) space for every
 $0 < p \leq \infty$.

 Returning to general symmetric spaces $\mathbf{E} = \mathbf{E}(\Omega,\mathcal{F}, \mu)$ we
  can notice:
\begin{itemize}
 \item[$\bullet$]
  Let $\mathbf{E}$ be a symmetric Banach space on $(\Omega,\mathcal{F}, \mu)$.
  Then there exist continuous embeddings
$$
  \mathbf{L}_1\cap \mathbf{L}_\infty \subseteq \mathbf{E} \subseteq \mathbf{L}_1+ \mathbf{L}_\infty  \subset \mathbf{L}_0
$$
with
$$
\varphi_\mathbf{E}(1)\|\cdot\|_{\mathbf{L}_1\cap \mathbf{L}_\infty } \geq \|\cdot\|_\mathbf{E} \geq \varphi_\mathbf{E}(1)\|\cdot\|_{\mathbf{L}_1+ \mathbf{L}_\infty },
$$
 where $\varphi_\mathbf{E}(t) := \|1_{[0,t]}\|, \ t \geq 0$ is the fundamental function of $\mathbf{E}$ and
 $\varphi_\mathbf{E}(1) = \|1_{[0,1]}\|$.
  (See, \cite{KrPiSe}, \cite{MuRu}.)
 \item[$\bullet$]
   Let $\mathbf{E}$ be a symmetric quasi-Banach space on $(\Omega,\mathcal{F}, \mu)$
   which satisfies the weak triangle inequality
$$
  \|f+g\|_\mathbf{E} \leq C (\|f\|_\mathbf{E }+ \|g\|_\mathbf{E}) \;,\; f,g \in \mathbf{E}
$$
with a constant $C >1$.

  Then by the Aoki-Rolewicz theorem ([Ao], [Rol]), the quasi-norm $\|\cdot\|_\mathbf{E}$ can be
 replaced by an equivalent $p$-sub-additive quasi-norm $|\|\cdot|\|_\mathbf{E}$ such that
 $$
  |\|f+g|\|_\mathbf{E}^p \leq |\|f|\|_\mathbf{E}^p  + |\|g|\|_\mathbf{E}^p  \;,\; f,g \in  \mathbf{E}\;,
 $$
 where
 $$
  p := \frac{\ln 2}{\ln 2 + \ln C} < 1 \;.
 $$
    Such the quasi-norm  $|\|\cdot|\|_E$ is called a {\it $p$-norm} and the quasi-Banach space $\mathbf{E}$
 is said to be {\it  $p$-normable}.
   Note that $\mathbf{E}$ becomes a complete linear metric space with respect to the translation-invariant metric
 $\delta_\mathbf{E}(f, g) = |\|f-g |\|_\mathbf{E}^p, \ f,g \in \mathbf{E}$.

    The constant $C_\mathbf{E} = \inf C \geq 1$ defined by $C$ in the weak triangle
 inequality, is called the {\it modulus of concavity} of the quasi-Banach space $\mathbf{E}$ ([Kal]).
    Evidently, $C_\mathbf{E} =1$ if $\mathbf{E}$ is normable, however, the converse is not true in general.
\end{itemize}

\subsection{Equimeasurable symmetric spaces}
 \label{ss Equimeasurable symmetric spaces}

Recall that here and throughout, the considered measure spaces $(\Omega,\mathcal{F}, \mu)$ are assumed to have
 finite or infinite $\sigma$-finite {\it non-atomic} measures $\mu$.

     The corresponding to $(\Omega,\mathcal{F}, \mu)$ standard measure space $(I, \mathcal{B}_m, m)$ is defined as
 $I=[0, \infty)$ if $\mu(\Omega) =  \infty $ and $I=[0, a]$ if $\mu(\Omega) = a < \infty $.
     Here, $m$ is the usual Lebesgue measure on $I$ and $\mathcal{B}_m$ is the $m$-completion of the Borel
 $\sigma$-algebra $\mathcal{B} = \mathcal{B}(I)$ with respect to $m$.

  A symmetric space $\mathbf{E}(\Omega,\mathcal{F}, \mu)$ will be called {\it standard} if
 the corresponding measure space $(\Omega,\mathcal{F}, \mu)$ is standard.

   For a symmetric space $\mathbf{E} = \mathbf{E}(\Omega,\mathcal{F},\mu)$  we set
$$
     \Xi(\mathbf{E}) := \{ \xi_{f,\mu} \in \mathbf{L}_0(I, \mathcal{B}_m, m) \colon f \in \mathbf{L}_0(\Omega,\mathcal{F},\mu)  \}.
$$
  This additional condition $\xi_{f,\mu} \in \mathbf{L}_0(I, \mathcal{B}_m, m)$, i.e. $f \in \mathbf{L}_0^\xi(\Omega, \mathcal{F}_\mu, \mu)$ makes sense if $\mu(\Omega) = \infty$.

 The function $\|\cdot\|_\mathbf{E} \colon\mathbf{E} \to [0,\infty)$ induces the mapping
 $$
 \|\cdot\|_{\Xi(\mathbf{E})} \colon \Xi(\mathbf{E}) \ni \xi_{f,\mu} \to \|\xi_{f,\mu}\|_{\Xi(\mathbf{E})} =
  \|f\|_\mathbf{E} \in [0,\infty) \; ,\; f \in \mathbf{E} \;.
 $$

    Two symmetric spaces $\mathbf{E}_1 = \mathbf{E}_1(\Omega_1,\mathcal{F}_1, \mu_1)$ and $\mathbf{E}_2 = \mathbf{E}_2(\Omega_2,\mathcal{F}_2,\mu_2)$
 will be called {\it equimeasurable} if
 $$
 \Xi(\mathbf{E}_1) = \Xi(\mathbf{E}_2).
 $$

    If, in addition, $\|\cdot\|_{\Xi(\mathbf{E}_1)} = \|\cdot\|_{\Xi(\mathbf{E}_2)}$, the spaces
 $\mathbf{E}_1 = \mathbf{E}_1(\Omega_1,\mathcal{F}_1,\mu_1)$ and $\mathbf{E}_2 = \mathbf{ E}_2(\Omega_2,\mathcal{F}_2,\mu_2)$ will be called
 {\it strictly equimeasurable}.

     The following two theorems show that each class of equimeasurable symmetric spaces contains a
 standard symmetric space, while all equimeasurable standard symmetric spaces coincide.

\begin{theorem}
 \label{th main theorem 1}
     Let $\mathbf{ E}(I,\mathcal{B}_m,m)$ be a standard symmetric space and $(\Omega,\mathcal{F}, \mu)$ be an arbitrary
 (finite or infinite $\sigma$-finite non-atomic) measure space $(\Omega,\mathcal{F}, \mu)$, and  let
\begin{equation}
 \label{eq E main theorem 1}
  \mathbf{E}(\Omega,\mathcal{F}, \mu) := \{f \in \mathbf{L}_0(\Omega,\mathcal{F}, \mu)  \colon  \xi_{f,\mu}  \in \mathbf{E}(I,\mathcal{B}_m,m) \}
\end{equation}
 and
\begin{equation}
 \label{eq norm main theorem 1}
   \|f\|_{\mathbf{E}(\Omega,\mathcal{F}, \mu)} = \|\xi_{f,\mu}\|_{\mathbf{E}(I,\mathcal{B}_m,m)} \;\;,\;\; f \in \mathbf{E}(\Omega,\mathcal{F}, \mu)
\end{equation}
   Then the space $(\mathbf{E}(\Omega,\mathcal{F}, \mu),\|\cdot\|_{\mathbf{E}(\Omega,\mathcal{F}, \mu)})$ defined by
 (\ref{eq E main theorem 1}) and (\ref{eq norm main theorem 1}) is a symmetric space on $(\Omega,\mathcal{F}, \mu)$.
   The symmetric spaces $\mathbf{E}(I,\mathcal{B}_m,m)$ and $\mathbf{E}(\Omega,\mathcal{F}, \mu)$ are strictly equimeasurable.
\end{theorem}

 \begin{theorem}
  \label{th main theorem 2}
  Let $\mathbf{E}(\Omega,\mathcal{F}, \mu)$ be a symmetric space on a measure space  $(\Omega,\mathcal{F}, \mu)$, and let
\begin{equation}
 \label{eq E main theorem 2}
  \mathbf{E}(I,\mathcal{B}_m,m) := \{g \in \mathbf{L}_0(I,\mathcal{B}_m,m) \colon \xi_{g,m}  = \xi_{f,\mu} \; \mbox{for some}  \; f\in (\Omega,\mathcal{F}, \mu)  \}\,,
\end{equation}
 and
\begin{equation}
 \label{eq norm main theorem 2}
 \|g\|_{\mathbf{E}(I,\mathcal{B}_m,m)} = \|f\|_{\mathbf{E}(\Omega,\mathcal{F}, \mu)} \;\;\mbox{if}
 \;\; \xi_{g,m} = \xi_{f,\mu} \;\;\mbox{and}\;\; f \in \mathbf{E}(\Omega,\mathcal{F}, \mu)\;.
\end{equation}
   Then the space $ (\mathbf{E}(I,\mathcal{B}_m,m),\|\cdot\|_{\mathbf{E}(I,\mathcal{B}_m,m)})$ defined by
 (\ref{eq E main theorem 2}) and (\ref{eq norm main theorem 2}) is a standard symmetric space.
   The symmetric spaces $\mathbf{E}(\Omega,\mathcal{F}, \mu)$ and  $ \mathbf{E}(I,\mathcal{B}_m,m)$ are strictly equimeasurable.
\end{theorem}

    We will consider these theorems separately for separable and nonseparable measure spaces
 $(\Omega,\mathcal{F}, \mu)$.

 Recall that a measure space $(\Omega,\mathcal{F}, \mu)$ is called separable if the $\sigma$-algebra $\mathcal{F}$ is
 countably generated mod $\mu$.
     This means that there exists a countable subalgebra $\mathcal{A} \subset \mathcal{F}$ such that
 $$
 \mathcal{F} \subseteq (\mathcal{F}(\mathcal{A}))_\mu \subseteq \mathcal{F}_{\mu}
 $$
  where $ (\mathcal{F}(\mathcal{A}))_\mu$ is the $\mu$-completion of the
 $\sigma$-algebra $ \mathcal{F}(\mathcal{A})$ generated by the subalgebra $\mathcal{A}$.
    The countable subalgebra $\mathcal{A} $ is dense in $\mathcal{F}$ in the following sense: for every $A \in \mathcal{F}$ and
 every $\varepsilon >0$ there is $B \in \mathcal{A}$ such that
 $$
 \mu(A\triangle B)< \varepsilon.
 $$

 \begin{theorem}
  \label{th isomorphism Phi Separable case}
     Let $(\Omega,\mathcal{F}, \mu)$ be a separable $\sigma$-finite non-atomic measure space and
 $(I,\mathcal{B}_m, m)$ be the corresponding standard measure space.
     Then there exists an algebraic, order, topological and integral preserving isomorphism
 $$
 \Phi  \colon \mathbf{L}_0(\Omega,\mathcal{F}, \mu) \to   \mathbf{L}_0(I,\mathcal{B}_m, m)\;.
 $$
       For every symmetric space $\mathbf{E}(\Omega,\mathcal{F}, \mu)\subset \mathbf{L}_0(\Omega,\mathcal{F}, \mu)$ the restriction
 $\Phi_\mathbf{E} = \Phi|_{\mathbf{E}(\Omega,\mathcal{F}, \mu)}$ is an isometric isomorphism between the symmetric space
 $\mathbf{E}(\Omega,\mathcal{F}, \mu)$ and its equimeasurable standard symmetric space $\mathbf{E}(I,\mathcal{B}_m, m)$.
 \end{theorem}

     We prove the theorem in Section \ref{s Symmetric spaces on separable measure spaces}, where an
 explicit construction of the isomorphism $\Phi$ is given.

     It is straightforward that Theorem \ref{th isomorphism Phi Separable case} implies Theorems
 \ref{th main theorem 1} and \ref{th main theorem 2} in the separable case.

\section{Symmetric spaces on separable measure spaces}
\label{s Symmetric spaces on separable measure spaces}

\subsection{Lebesgue spaces.}
 \label{ss Lebesgue spaces}

    Recall that a measure space $(\Omega,\mathcal{F}, \mu)$ is called separable if the $\sigma$-algebra $\mathcal{F}$ is
 countably generated mod $\mu$.
    This means that there exists a countable subalgebra $\mathcal{A} \subseteq \mathcal{F}$ such that
 $\mathcal{F} \subseteq (\mathcal{F}(\mathcal{A}))_\mu$, where $(\mathcal{F}(\mathcal{A}))_\mu$ is the $\mu$-completion of the
 $\sigma$-algebra $ \mathcal{F}(\mathcal{A})$ generated by the subalgebra $\mathcal{A}$, \ $(\mathcal{F}(\mathcal{A}))_\mu \subseteq \mathcal{F}_\mu$ .

    We consider the $\sigma$-algebra $\mathcal{F}_\mu$ as an abstract Boolean algebra $\nabla = \nabla (\Omega,\mathcal{F}_\mu,\mu)$
 equipped with the metric
$$
    \delta_\nabla(A,B)= \delta_0(1_A,1_B)  \;,\; A,B \in \nabla \;.
$$
    By the definition, the natural imbedding
    $ \nabla \in A \to 1_A \in  \mathbf{L}_0$
     is isometric,
  whence $(\nabla, \delta_\nabla)$ is a complete metric space as well as $(\mathbf{L}_0,\delta_0)$.

  Recall also the following known result:

   \begin{theorem}
 \label{th separability of gW, nabla and L_0}
   The following conditions are equivalent:
 \begin{itemize}
  \item[1.]
   $(\Omega, \mathcal{F}_\mu,\mu)$  is separable,
  \item[2.]
   $(\nabla, \delta_\nabla)$ is separable,
  \item[3.]
   $(\mathbf{L}_0, \delta_0)$ is separable.
 \end{itemize}
 \end{theorem}
  See, for instance, \cite{KanAk}, $\S 1.6.$

Let $\mathcal{A}$ be a countable subalgebra of $\mathcal{F}_\mu$.
     We say that $\mathcal{A}$ separates points of $\Omega$ if for every pair  $\omega_1, \omega_2 \in \Omega$ there is
 $A \in \mathcal{A}$ such that $\omega_1 \in A$ and $\omega_2  \notin A$.

     If $\mathcal{A}$ does not separates points of $\Omega$ we can use the partition $\zeta=\zeta(\mathcal{A})$ of $\Omega$
 generated by $\mathcal{A}$,  where
 $$
   \omega_1\stackrel{\zeta}{\sim} \omega_2 \Longleftrightarrow \mbox{there is no}\;
   A \in \mathcal{A} \;\mbox{such that} \; \omega_1 \in A \;\mbox{and}\; \omega_2  \notin  A \;.
 $$
     The partition $\zeta$ consists of elements $\zeta(\omega)$, $\omega\in\Omega$, where $\zeta(\omega)$ is intersection of all $A$ such that $\omega\in A\in\mathcal{A}$.

A subset $F\subseteq \Omega$ is called a $\zeta$-set if $\zeta(\omega)\subseteq F$ for all $\omega\in F$

     We denote by $\mathcal{F}(\mathcal{A})^\mu$ the $\sigma$-algebra of all $\mu$-measurable  $\zeta$-sets.
     Obviously,
$$
 \mathcal{A} \subseteq \mathcal{F}(\mathcal{A}) \subseteq \mathcal{F}(\mathcal{A})^\mu \subseteq \mathcal{F}(\mathcal{A})_\mu \subseteq \mathcal{F}_\mu \;.
$$
 Clearly,  $\mathcal{F}(\mathcal{A})^\mu = \mathcal{F}(\mathcal{A})_\mu$ if $\mathcal{A}$ separates points of $\Omega$.

    A countable subalgebra $\mathcal{A} $ of  $\mathcal{F}$ is said to be {\it a countable base} of a measure space
 $(\Omega, \mathcal{F},\mu)$ if
 \begin{itemize}
  \item[$\bullet$]
   $\mathcal{A}$ generates $\mathcal{F}$  \ mod $\mu$, i.e. $\mathcal{F} \subseteq \mathcal{F}(\mathcal{A})_\mu$.
  \item[$\bullet$]
   $\mathcal{A}$ separates points of $\Omega$ \ mod $\mu$, i.e. $\mathcal{A}$ separates the points of some measurable
   subset $\Omega_0 \subseteq \Omega$ of full measure $\mu$.
 \end{itemize}

  Recall that a class $\mathcal{E}$ of subsets of $\Omega$ is called {\it a compact class} if for every sequence
 $\{C_n\}_{n\geq 1} \subset \mathcal{E}$,
$$
 \bigcap_{n\geq 1} C_n = \emptyset \;\Longleftrightarrow \; \bigcap_{n\leq N} C_n = \emptyset \;\mbox{for some}\; N \in \mathbf{N} \;,
$$
see, \cite{N}, sec. 1.6.

     A measure space $(\Omega, \mathcal{F},\mu)$ is called {\it a Lebesgue space} if there exists a countable base
 $\mathcal{A}$ of  $(\Omega, \mathcal{F},\mu)$ which is a compact class \ mod $\mu$, i.e. $\mathcal{A}$ is a compact class on some
 measurable subset $\Omega_0 \subseteq \Omega$ of full measure $\mu$.

We use here the definition of Lebesgue spaces from \cite{ViRuFe}.
   It is a little different but equivalent to the original definition from \cite{Ro$_1$}.

We need the following known result:
 \begin{itemize}
  \item[$\bullet$]
     Let $(X,d)$ be a Polish (complete separable metric) space and $\lambda$ be a Borel measure on $X$.
     The measure $\lambda$ is defined on the the $\sigma$-algebra $\mathcal{B}=\mathcal{B}(X)$ of all Borel subsets of $X$ and
   extended to the $\lambda$-completion $\mathcal{B}_\lambda$.
     Then the measure space $(X,\mathcal{B}_\lambda, \lambda)$ is a Lebesgue space.
     In particular, every standard space $(I,\mathcal{B}_m,m)$ is a Lebesgue space.
  \item[$\bullet$]
     Let $(\Omega, \mathcal{F},\mu)$ be a non-atomic Lebesgue space.
     Then there exists a measure preserving isomorphism $\phi$ between $(\Omega, \mathcal{F},\mu)$ and the
   corresponding standard space $(I,\mathcal{B}_m,m)$.
 \end{itemize}

 Note that there is no need for the usual correction \ mod $\mu$ in the latter statement since the
 both considered measure spaces are non-atomic.

   As a consequence we have:
  \begin{itemize}
  \item[$\bullet$]
    Let $\phi \colon (\Omega, \mathcal{F},\mu) \to (I,\mathcal{B}_m,m)$ be a measure preserving isomorphism between
 a non-atomic Lebesgue space $(\Omega,\mathcal{F},\mu)$ and the corresponding standard space $(I,\mathcal{B}_m,m)$.
    Then
 $$
 \Phi_\phi  \colon \mathbf{L}_0(\Omega,\mathcal{F},\mu) \ni f \to  f \circ \phi^{-1} \in \mathbf{L}_0(I,\mathcal{B}_m, m)
 $$
 is an algebraic, order, topological and integral preserving isomorphism desired in
 Theorem \ref{th isomorphism Phi Separable case}.
\end{itemize}

\subsection{Subspaces and factor-spaces}
 \label{ss Subspaces and factor-spaces}

Let $(\Omega, \mathcal{F},\mu)$ be a finite or infinite $\sigma$-finite measure space and $\Omega_0 \subseteq \Omega$ be
 an arbitrary (not necessarily measurable) subset of $\Omega$.

   We set
$$
     \mathcal{F}_0 := \{ A \cap \Omega_0 \colon A \in \mathcal{F}\}
$$
 and
$$
   \mu_0(A_0) = \mu(A) \;\mbox{if}\; A_0 = A \cap \Omega_0 \;,\; A  \in \mathcal{F} \;.
$$

   Let us assume that the subset $\Omega_0$ is of full outer measure $\mu^*$ in $\Omega$.
   Then the following statements hold:

\begin{itemize}
 \item[$\bullet$]
  $(\Omega_0, \mathcal{F}_0,\mu_0)$ is a measure space.
  The measure $\mu_0$ is finite, infinite $\sigma$-finite, or non-atomic iff the measure $\mu$ is so.
 \item[$\bullet$]
  $(\Omega, \mathcal{F},\mu)$ is a Lebesgue space iff $(\Omega_0, \mathcal{F}_0,\mu_0)$ is a Lebesgue
  space and $\Omega_0$ is $\mu$-measurable.
\end{itemize}

 We need also the following important fact:
  \begin{itemize}
   \item[$\bullet$]
    The restriction $ \Phi_0 \colon f \to f|_{\Omega_0}$ induces an algebraic, order, topological and integral
    preserving isomorphism between $ \mathbf{L}_0(\Omega,\mathcal{F},\mu)$ and $ \mathbf{L}_0(\Omega_0,\mathcal{F}_0,\mu_0)$.
  \end{itemize}
The latter result (being trivial if $\Omega_0$ is measurable) holds also for non-measurable subsets
 $\Omega_0$ of full outer measure $\mu^*$ in $\Omega$, see, for instance, \cite{Fr}, sec. 214 A-J.

Let  $(\Omega,\mathcal{F},\mu)$ be a complete separable measure space, $\mathcal{F} = \mathcal{F}_\mu$.
     Let the $\sigma$-algebra $\mathcal{F}$ is generated (mod $\mu$) by a countable subalgebra $\mathcal{A} \subset \mathcal{F}$,
 i.e. $\mathcal{F} \subseteq (\mathcal{F}(\mathcal{A}))_\mu$, where $(\mathcal{F}(\mathcal{A}))_\mu$ is the $\mu$-completion of the $\sigma$-algebra
  $ \mathcal{F}(\mathcal{A})$ generated by the subalgebra $\mathcal{A}$.

     Consider the partition $\zeta=\zeta(\mathcal{A})$ of $\Omega$ generated by
 the subalgebra $\mathcal{A}$ and the $\sigma$-algebra  $\mathcal{F}(\mathcal{A})^\mu$ of all $\mu$-measurable $\zeta$-sets.

     Recall that the factor-space $(\Omega/\zeta, \mathcal{F}/\zeta,\mu/\zeta)$ is defined by
$$
    \mathcal{F}/\zeta = \{A' \subseteq \Omega/\zeta \colon \pi_\zeta^{-1}(A') \in \mathcal{F}(\mathcal{A})^\mu\}
$$
 and
$$
   \mu/\zeta(A) = \mu(\pi_\zeta^{-1}(A')) \;,\; A' \in  \mathcal{F}/\zeta \;,
$$
 where $\pi_\zeta \colon \Omega \to \Omega/\zeta $
 denotes the natural projection.

    By the definition, the projection $\pi_\zeta$ is measure preserving and the restriction
 $$
      \pi_\zeta|_{\mathcal{F}(\mathcal{A})^\mu} \colon \mathcal{F}(\mathcal{A})^\mu \ni A \to \pi_\zeta(A) \in \mathcal{F}/\zeta
 $$
 is a measure preserving isomorphism between the $\sigma$-algebras.

  Moreover, by the definition,  the countable subalgebra $\mathcal{A}' = \pi_\zeta(\mathcal{A})$ is a countable base of
 the separable measure space $(\Omega/\zeta, \mathcal{F}/\zeta,\mu/\zeta)$.

     On the other hand, since $(\mathcal{F}(\mathcal{A})^\mu)_\mu = \mathcal{F}(\mathcal{A})_\mu =\mathcal{F}_\mu$, the mapping $\pi_\zeta$ induces
 an algebraic, order, topological and integral preserving isomorphism $\Phi_\zeta$ between
 $ \mathbf{L}_0(\Omega,\mathcal{F}_\mu,\mu)$ and $ \mathbf{L}_0(\Omega/\zeta,\mathcal{F}/\zeta,\mu/\zeta)$ such that
 $$
   \Phi_\zeta^{-1} \colon \mathbf{L}_0(\Omega/\zeta,\mathcal{F}/\zeta,\mu/\zeta) \ni f \to f  \in \mathbf{ L}_0(\Omega,\mathcal{F},\mu)\;.
 $$

\subsection{The Stone representation}
 \label{ss The Stone representation}

     We will show now that every measure space $(\Omega,\mathcal{F},\mu)$ with a countable base $\mathcal{A}$ can be
 identified with subspace of full outer measure in a suitable Lebesgue space
 $(\mathcal{S},\mathcal{B}(\mathcal{S})_{\mu_\mathcal{S}},\mu_\mathcal{S})$.
     We may assume without loss of the generality that $\mu(\Omega) =1$ by passing to a probability measure
 $\mu'$ which is equivalent to $\mu$.

  To this aim we consider the set algebra $\mathcal{A}$ as an abstract Boolean algebra and use its Stone representation (see. \cite{St}).

  Recall that the Stone space $\mathcal{S} = \mathcal{S}(\mathcal{A})$ consists of all ultrafilters $U$ of the Boolean algebra
  $\mathcal{A}$.
  A subset $U \subseteq \mathcal{A}$ is called a ultrafilter if for all $A,B \in \mathcal{A}$,

\begin{itemize}
 \item[$\bullet$]
  $\;\emptyset \notin U$.
 \item[$\bullet$]
  $\;   B \in U, \   B\subseteq A  \Longrightarrow  A \in U$.
 \item[$\bullet$]
  $\;   A \in U,  B \in U  \Longrightarrow  A \cap B \in U$.
 \item[$\bullet$]
  $\; A \in U \Longrightarrow A^\complement \notin U $, where $A \rightarrow A^\complement$ denotes the complement operation in $\mathcal{A}$.
\end{itemize}

     We set $u_\mathcal{A}(A) = \{ U \in \mathcal{S} \colon A \in U \}$ and define the Stone representation
 $$
 u^\mathcal{A} \colon  \mathcal{A} \ni A \to  u^\mathcal{A}(A) \in u_\mathcal{A}(\mathcal{A})\subseteq \mathcal{S}
 $$
 of the Boolean algebra $\mathcal{A}$ as the set algebra $u_\mathcal{A}(\mathcal{A})$ on the space $\mathcal{S} = \mathcal{S}(\mathcal{A})$.

  The set algebra $u^\mathcal{A}(\mathcal{A})$ separates points in $\mathcal{S}$ and it can be taken as a base of a natural
  Stone topology on $\mathcal{S}$.

   We shall use the Stone construction only in the case when the Boolean algebra $\mathcal{A}$ is countable.
  In this case:
\begin{itemize}
 \item[$\bullet$]
  $\; \mathcal{S}$ is a compact separable totally disconnected metrizable topological space.
 \item[$\bullet$]
  $\; u^\mathcal{A}$ is an algebraic isomorphism between the Boolean algebras $\mathcal{A}$ and the algebra $\mathcal{K}(\mathcal{S})$
  of all open-closed subsets of $\mathcal{S}$.
   \item[$\bullet$]
    For every $\omega \in \Omega$ the set $ U_\omega := \{A \in \mathcal{A} : \omega \in A \} \subset \mathcal{A}$
    is a ultrafilter
    and the corresponding mapping $ \phi_\mathcal{A} \colon \Omega \ni \omega \to U_\omega \in \mathcal{S}$ is injective.
   \item[$\bullet$]
    $\;\phi_\mathcal{A} (A) =u^\mathcal{A}(A)$ for all  $A \in \mathcal{A}$ and the mapping $ \phi_\mathcal{A}$ is bijective iff $\mathcal{A}$
    is a compact class on $\Omega$.
\end{itemize}

 Further, since $\mathcal{K}=\mathcal{K}(\mathcal{S})$ is a compact class in $\mathcal{S}$ we have
\begin{itemize}
  \item[$\bullet$]
   The measure $\mu \circ \phi^{-1}$ on $\mathcal{K}$ can be extended to a Borel measure $\mu_\mathcal{S}$ on $\mathcal{B}(\mathcal{S})$
   and the completion $(\mathcal{S}, \mathcal{B}(\mathcal{S})_{\mu_\mathcal{S}},\mu_\mathcal{S})$ is a Lebesgue space with a compact base $\mathcal{K}$.
  \item[$\bullet$]
    The embedding $\phi_\mathcal{A}\colon \Omega \to \mathcal{S}$ is measure preserving and $\phi_\mathcal{A}(\Omega)$ is of full
    outer measure $\mu_\mathcal{S}^*$ in $\mathcal{S}$.
  \item[$\bullet$]
   The image $\phi_\mathcal{A}(\Omega)$  is
 measurable iff $(\Omega,\mathcal{F},\mu)$  is a Lebesgue space.
\end{itemize}

  Since the image $\phi_{\mathcal{A}}(\Omega)$  is of full outer measure $\mu^*$, the mapping $\phi_\mathcal{A}$ induces an algebraic, order, topological and integral preserving
 isomorphism $\Phi_\mathcal{A}$ between $ \mathbf{L}_0(\Omega,\mathcal{F}_\mu,\mu)$ and $(\mathcal{S}, \mathcal{B}(\mathcal{S})_{\mu_\mathcal{S}},\mu_\mathcal{S})$ such that
 $$
   \Phi_\mathcal{A}^{-1} \colon (\mathcal{S}, \mathcal{B}(\mathcal{S})_{\mu_\mathcal{S}},\mu_\mathcal{S}) \ni f \to
   f \circ \phi_\mathcal{A} \in \mathbf{L}_0(\Omega,\mathcal{F},\mu)\;.
 $$

\begin{remark}
    The Stone representation may be applied for any (not necessarily countable) subalgebras $\mathcal{A}$ of $\mathcal{F}$.
    We, however, use it only in the case when the Boolean algebra is countable.

    If $\mathcal{A}$ is not countable the corresponding compact totally disconnected space $ \mathcal{S}$  may be non-separable and non-metrizable.
    In this case, the corresponding measure space $(\mathcal{S}, \mathcal{B}(\mathcal{S})_{\mu_{\mathcal{S}}},\mu_{\mathcal{S}})$ is not a Lebesgue space, although $\mathcal{K}=\mathcal{K}(\mathcal{S})$ remains
 to be a compact class and a non-countable base of $\mathcal{S}$.
\end{remark}

\subsection{Construction of $\Phi$}
 \label{ss Construction of $Phi$}

Combining the above results we can construct the isomorphism  $\Phi$ desired in
  Theorem \ref{th isomorphism Phi Separable case}.

      Indeed, let  $(\Omega, \mathcal{F},\mu)$ be a separable measure space with finite or infinite $\sigma$-finite
  non-atomic measure $\mu$ and let $(I,\mathcal{B}_m, m)$ be the corresponding standard measure space.
      Passing to the $\mu$-completion we may assume that $\mathcal{F} = \mathcal{F}_\mu$.

      We choose and fix a countable subalgebra $\mathcal{A} \subset \mathcal{F}$ which generates $\mathcal{F}$ (mod $\mu$),
  and define a measure preserving mappings:
\begin{itemize}
  \item[$\bullet$]
   $\; \pi_\zeta\colon (\Omega, \mathcal{F},\mu) \to (\Omega/\zeta, \mathcal{F}/\zeta,\mu/\zeta)$, where $\zeta=\zeta(\mathcal{A})$ is the
   partition of $\Omega$ generated by the algebra $\mathcal{A}$.
  \item[$\bullet$]
   $\;  \phi_{\mathcal{A}'} \colon (\Omega/\zeta, \mathcal{F}/\zeta,\mu/\zeta) \to (\mathcal{S}, \mathcal{B}(\mathcal{S})_{\mu_\mathcal{S}},\mu_\mathcal{S})$, where
   $\mathcal{A}' = \pi_\zeta(\mathcal{A})$ is a countable base of the factor-space $(\Omega/\zeta, \mathcal{F}/\zeta,\mu/\zeta)$.
  \item[$\bullet$]
   $\; \phi_{S,I} \colon (\mathcal{S}, \mathcal{B}(\mathcal{S})_{\mu_\mathcal{S}},\mu_\mathcal{S}) \to (I,\mathcal{B}_m, m)$, which an isomorphism
   between these Lebesgue spaces.
\end{itemize}

Using the composition $\pi_\zeta \circ \phi_{\mathcal{A}'} \circ \phi_{S,I}$ we have proved
 \begin{theorem}
 \label{th isomorphism phi_cA}
    Let $(\Omega, \mathcal{F}_\mu,\mu)$ be a separable $\sigma$-finite non-atomic measure space and $\mathcal{A}$ be a
 countable subalgebra of $\mathcal{F}_\mu$ such that $\mathcal{F}(\mathcal{A})_\mu = \mathcal{F}_\mu$.
    Let $(I, \mathcal{B}(I)_m, m)$ be the corresponding to $(\Omega,\mathcal{F},\mu)$ standard measure space.

     Then there exist a subset $J$ of $I$ of full outer measure $m^*$ in $I$ and a measure preserving
  mapping
$$
        \phi_\mathcal{A}\colon (\Omega, \mathcal{F}_\mu, \mu) \to (J, \mathcal{F}_J, m_J)
$$
 such that the restriction
$$
 \phi|_{\mathcal{F}^\mu(\mathcal{A})} \colon (\Omega, \mathcal{F}^\mu(\mathcal{A}), \mu|_{\mathcal{F}^\mu(\mathcal{A})}) \to (J, \mathcal{F}_J, m_J)
$$
 is a measure preserving isomorphism between the space $(\Omega, \mathcal{F}^\mu(\mathcal{A}), \mu|_{\mathcal{F}^\mu(\mathcal{A})}) $ and
 the measure subspace $(J,\mathcal{F}_J, m_J)$ induced by $(I, \mathcal{B}(I)_m, m)$ on $J$.
\end{theorem}

 The measure preserving mapping $\phi_\mathcal{A}$ determines the algebraic, order, topological and integral
 preserving isomorphism
 $$
 \Phi_\mathcal{A} \colon \mathbf{L}_0(\Omega,\mathcal{F},\mu) \to \mathbf{L}_0(I,\mathcal{B}_m, m) \;,
 $$
 desired in Theorem \ref{th isomorphism Phi Separable case}.

   The proposition requires some explanation.
 \begin{itemize}
  \item[$\bullet$]
   The separable measure spaces $(\Omega,\mathcal{F},\mu)$ and $(\Omega/\zeta, \mathcal{F}/\zeta,\mu/\zeta)$ are not isomorphic
   if the subalgebra $\mathcal{A}$ does not separates the points of $\Omega$.
   Nevertheless, the spaces $\mathbf{L}_0(\Omega,\mathcal{F},\mu) = \mathbf{L}_0(\Omega,\mathcal{F}^\mu,\mu)$ and
   $\mathbf{L}_0(\Omega/\zeta, \mathcal{F}/\zeta,\mu/\zeta)$ can be identified by means of the projection $\pi_\zeta$.
  \item[$\bullet$]
   The separable measure spaces $(\Omega/\zeta, \mathcal{F}/\zeta,\mu/\zeta)$ and $(J, \mathcal{F}_J, m_J)$ are isomorphic.
   They have countable bases but need not be Lebesgue spaces in general.
  \item[$\bullet$]
    The subset $J \subseteq I$ need not be measurable and $(J,\mathcal{F}_J, m_J)$ need not be a Lebesgue space.
    Nevertheless, the spaces $\mathbf{L}_0(I,\mathcal{B}_m, m)$ and $\mathbf{L}_0(J,\mathcal{F}_J, m_J)$ can be identified by means of
    the restriction $f \to f|_J$.
   \end{itemize}

\section{Symmetric spaces on general measure spaces}
\label{s Symmetric spaces on general measure spaces}

\subsection{The proof of Theorems \ref{th main theorem 1}}
 \label{ss The proof of theorems 1}

    Let $(\Omega, \mathcal{F}_\mu, \mu)$ be a general (not necessarily separable) non-atomic measure space and
 let $(I,\mathcal{B}_m, m)$ be the corresponding standard measure space.

    Let use denote by $\mathfrak{A}$ the class of all subalgebras $\mathcal{A} \subset \mathcal{F}$ such
   that
 \begin{itemize}
   \item[$\bullet$]
    The subalgebras $\mathcal{A}$ is countable,
   \item[$\bullet$]
    The corresponding separable measure space $(\Omega, \mathcal{F}(\mathcal{A})_\mu, \mu|_{\mathcal{F}(\mathcal{A})_\mu})$ is non-atomic,
   \item[$\bullet$]
    The restricted measure $\mu|_{\mathcal{F}(\mathcal{A})_\mu}$ is $\sigma$-finite.
  \end{itemize}

  Recall that the measure $\mu$ itself is assumed to be non-atomic and $\sigma$-finite.

      To reduce notation we write: $ \mathcal{F}_\mathcal{A} = \mathcal{F}(\mathcal{A})_\mu, \  \mu_\mathcal{A} = \mu|_{\mathcal{F}(\mathcal{A})_\mu}$
 for $\mathcal{A} \in \mathfrak{A}$.

For every $\mathcal{A} \in \mathfrak{A}$ we choose and fix an algebraic, order, topological and integral
 preserving isomorphism
 $$
 \Phi_\mathcal{A}  \colon \mathbf{L}_0(\Omega_\mathcal{A},\mathcal{F}_\mathcal{A},\mu_\mathcal{A}) \to \mathbf{L}_0(I,\mathcal{B}_m, m) \;,
 $$
 provided by  Theorem \ref{th isomorphism Phi Separable case}.

  Let $\mathbf{E}_0 = \mathbf{E}_0(I,\mathcal{B}_m, m)$ be a standard symmetric space and
 $$
 \Xi_0 = \Xi(\mathbf{E}_0) \;,\; \|\cdot\|_0 = \|\cdot\|_{\Xi(\mathbf{E}_0)} \;.
 $$
     Since $ \Phi_\mathcal{A}$ preserves distribution functions, the space $\mathbf{E}_\mathcal{A} := \Phi_\mathcal{A}^{-1}(\mathbf{E}_0)$ is a
 symmetric space on the separable measure space $(\Omega_\mathcal{A},\mathcal{F}_\mathcal{A},\mu_\mathcal{A})$, and
 $\mathbf{E}_\mathcal{A} = \mathbf{E}_\mathcal{A}(\Omega_\mathcal{A},\mathcal{F}_\mathcal{A},\mu_\mathcal{A})$ is the only symmetric space on $(\Omega_\mathcal{A},\mathcal{F}_\mathcal{A},\mu_\mathcal{A})$,
 which is equimeasurable with $\mathbf{E}_0$.
    So we have, $ (\Xi_\mathcal{A}, \|\cdot\|_\mathcal{A}) = (\Xi_0, \|\cdot\|_0)$ for each
 $\mathcal{A} \in \mathfrak{A}$.

 Turning to the measure space $(\Omega,\mathcal{F},\mu)$ consider the space
 $(\mathbf{E}_\mathcal{A}(\Omega,\mathcal{F},\mu),\|\cdot\|_{\mathbf{E}_\mathcal{A}(\Omega,\mathcal{F},\mu)})$ defined by
 (\ref{eq E main theorem 1}) and (\ref{eq norm main theorem 1}).

   Since the measure space $(\Omega,\mathcal{F},\mu)$ is non-atomic and
   $\sigma$-finite,
$$
      \mathcal{F}_\mu = \bigcup_{\mathcal{A} \in \mathfrak{A}}\mathcal{F}_\mathcal{A} \;,
$$
 and hence,
$$
     \mathbf{ L}_0(\Omega,\mathcal{F}_\mu, \mu) = \bigcup_{\mathcal{A} \in \mathfrak{A}}  \mathbf{L}_0(\Omega, \mathcal{F}_\mathcal{A}, \mu_\mathcal{A})\;.
$$
    This equality implies that
$$
      \mathbf{E}(\Omega, \mathcal{F}_\mu, \mu) = \bigcup_{\mathcal{A} \in \mathfrak{A}}  \mathbf{E}_\mathcal{A}(\Omega, \mathcal{F}_\mathcal{A},\mu_\mathcal{A})\;.
$$
      Here, each $\mathbf{E}_\mathcal{A}(\Omega, \mathcal{F}_\mathcal{A},\mu_\mathcal{A})$ is a symmetric space on its own measure space
 $(\Omega_\mathcal{A},\mathcal{F}_\mathcal{A},\mu_\mathcal{A})$, whence their union is a symmetric space on $(\Omega,\mathcal{F},\mu)$.
     Moreover,
$$
  \Xi(\mathbf{E}(\Omega, \mathcal{F}_\mu, \mu)) =\Xi(\mathbf{E}_\mathcal{A}(\Omega, \mathcal{F}_\mathcal{A},\mu_\mathcal{A})) =
  \Xi(\mathbf{E}_0(I,\mathcal{B}_m, m)) = \Xi_0  \;\mbox{for all}\; \mathcal{A}\in \mathfrak{A}  \;.
$$
     For every $f \in \mathbf{E}(\Omega, \mathcal{F}_\mu, \mu)$ there exists $\mathcal{A} \in \mathfrak{A}$ such that
  $f \in \mathbf{E}(\Omega, \mathcal{F}_\mathcal{A}, \mu_\mathcal{A})$ and
$$
 \|f\|_{\mathbf{E}(\Omega, \mathcal{F}_\mu, \mu)} =  \|f\|_{\mathbf{E}(\Omega,\mathcal{F}_\mathcal{A}, \mu_\mathcal{A})}=
 \|\xi_{f,\mu}\|_{\mathbf{E}_0(I,\mathcal{B}_m, m)} = \|\xi_{f,\mu}\|_0 \;.
$$
    Thus all the symmetric spaces $\mathbf{E}(\Omega, \mathcal{F}_\mu, \mu),  \mathbf{E}(\Omega, \mathcal{F}_\mathcal{A}, \mu_\mathcal{A})$
 and $\mathbf{E}_0(I,\mathcal{B}_m, m)$ are equimeasurable.

    The proof of Theorem \ref{th main theorem 1} is completed.

 \subsection{The proof of Theorems \ref{th main theorem 2}.}
 \label{ss The proof of theorems 2}

     Let now  $\mathbf{E}(\Omega, \mathcal{F}_\mu, \mu)$ be a symmetric space on a non-atomic finite or infinite $\sigma$-finite
 measure space $\mathbf{E}(\Omega, \mathcal{F}_\mu, \mu)$ and $(I,\mathcal{B}_m, m)$ be the corresponding standard measure space.
    Let also  $ \mathbf{E}(I,\mathcal{B}_m, m)$ and $\|\cdot\|_{\mathbf{E}(I,\mathcal{B}_m,m)}$ are defined by the equalities
 (\ref{eq E main theorem 2}) and (\ref{eq norm main theorem 2}).

 Choosing a fixed countable subalgebra $\mathcal{A}_1 \in \mathfrak{A}$ we consider the separable measure space
  $(\Omega, \mathcal{F}_1, \mu_1) = (\Omega, \mathcal{F}_{\mathcal{A}_1}, \mu_{\mathcal{A}_1})$ and the intersection
$$
    \mathbf{E}_1 =  \mathbf{E}_1(\Omega, \mathcal{F}_1, \mu_1) := \mathbf{E}(\Omega, \mathcal{F}_\mu, \mu) \cap  \mathbf{L}_0(\Omega, \mathcal{F}_{\mathcal{A}_1}, \mu_{\mathcal{A}_1}) \;.
$$
     Then the norm (or quasinorm) $ \|\cdot\|_{\mathbf{E}(\Omega, \mathcal{F}_\mu, \mu)}$ induces a norm (or quasinorm)
 $\|\cdot\|_1 = \|\cdot\|_{\mathbf{E}_1\Omega, \mathcal{F}_1, \mu_1)} $ on $\mathbf{E}_1$ and $(\mathbf{E}_1, \|\cdot\|_1)$ is a symmetric space on
 $(\Omega, \mathcal{F}_1, \mu_1) $.

     By Theorem \ref{th isomorphism Phi Separable case} there exists an algebraic, order, topological and integral preserving isomorphism
$$
 \Phi_1  \colon \mathbf{L}_0(\Omega, \mathcal{F}_1,\mu_1) \to   \mathbf{L}_0(I,\mathcal{B}_m, m)
$$
 such that the restriction $\Phi_1|_{\mathbf{E}_1}$ is an isometric isomorphism between the symmetric space $\mathbf{E}_1 =\mathbf{E}_1(\Omega, \mathcal{F}_1,\mu_1)$
 and its equimeasurable standard symmetric space $\mathbf{E}_1(I,\mathcal{B}_m, m)$.

    Obviously,
$$
 \mathbf{E}_1(\Omega, \mathcal{F}_1, \mu_1) \subseteq \mathbf{E}(\Omega, \mathcal{F}_\mu, \mu) \;\Longleftrightarrow\; \mathbf{E}_1(I,\mathcal{B}_m, m) \subseteq \mathbf{E}(I,\mathcal{B}_m, m) \;.
$$
      On the other hand, let $g \in \mathbf{E}(I,\mathcal{B}_m, m)$, i.e. $\xi_{g,m} = \xi_{f,\mu}$ for some $f \in \mathbf{E}(\Omega, \mathcal{F}_\mu, \mu)$.
      Since there exists $\mathcal{A} \in \mathfrak{A}$ such that $f \in \mathbf{L}_0(\Omega, \mathcal{F}_\mathcal{A}, \mu_\mathcal{A})$. we can use the isomorphism
$$
 \Phi_\mathcal{A}  \colon \mathbf{L}_0(\Omega, \mathcal{F}_\mathcal{A},\mu_\mathcal{A}) \to   \mathbf{L}_0(I,\mathcal{B}_m, m)
$$
 and set $f_1 = \Phi^{-1}_1 (\Phi_\mathcal{A} (f))$.
     Then $f_1 \in \mathbf{E}_1(\Omega, \mathcal{F}_1, \mu_1)$ and $\xi_{f_1,\mu} = \xi_{f,\mu} = \xi_{g,m}$, i.e. $g \in E_1(I,\mathcal{B}_m, m)$.

     Thus $\mathbf{E}(I,\mathcal{B}_m,m) = \mathbf{E}_1(I,\mathcal{B}_m, m)$. Whence $\mathbf{E}(I,\mathcal{B}_m,m)$ is a standard symmetric space since $\mathbf{E}_1(I,\mathcal{B}_m, m)$ is so.
     The spaces $\mathbf{E}(\Omega, \mathcal{F}_\mu, \mu), \ \mathbf{E}(I,\mathcal{B}_m, m)$ and $(\Omega, \mathcal{F}_\mathcal{A},\mu_\mathcal{A})$,  $\mathcal{A} \in \mathfrak{A}$, \ are strictly  equimeasurable.

    The proof of Theorem \ref{th main theorem 2} is completed.

\subsection{Non-separable measure spaces}
 \label{ss Non-separable measure spaces}

    Simple non-separable measure spaces $(\Omega, \mathcal{F},\mu)$ can be constructed in the following way.
$$
 (\Omega^\Upsilon,\mathcal{F}^\Upsilon,\mu^\Upsilon) = (I,\mathcal{B}_m, m) \times \prod_{\upsilon \in \Upsilon} (\Omega_\upsilon,\mathcal{F}_\upsilon,\mu_\upsilon)
$$
 where $(\Omega_\upsilon,\mathcal{F}_\upsilon,\mu_\upsilon) = (\{0,1\}, 2^{\{0,1\}}, (1/2, 1/2))$ for all $\upsilon \in \Upsilon$.

    Choosing the index sets $\Upsilon$ with various cardinalities, we obtain various measure spaces $ (\Omega^\Upsilon,\mathcal{F}^\Upsilon,\mu^\Upsilon)$ and
 various symmetric spaces $\mathbf{E}(\Omega^\Upsilon,\mathcal{F}^\Upsilon,\mu^\Upsilon)$ equimeasurable to the same standard symmetric spaces $\mathbf{E}(I,\mathcal{B}_m, m)$.

     On the other hand, for any countable subset $\Upsilon_0$ of $\Upsilon$,  the natural projection
$$
    \pi_0 \colon (\Omega^\Upsilon,\mathcal{F}^\Upsilon,\mu^\Upsilon) \to  (\Omega^{\Upsilon_0},\mathcal{F}^{\Upsilon_0},\mu^{\Upsilon_0})
    = (I,\mathcal{B}_m, m) \times \prod_{\upsilon \in \Upsilon} (\Omega_\upsilon,\mathcal{F}_\upsilon,\mu_\upsilon)
$$
 generates a separable measure space $(\Omega^\Upsilon,\mathcal{F}^{\Upsilon_0},\mu^{\Upsilon_0})$, where
 $\mathcal{F}^\Upsilon_0 = \{\pi_0^{-1} A_0, \ A_0 \in \mathcal{F}^{\Upsilon_0}\}$ and $\mu^\Upsilon_0 = \mu^\Upsilon|_{\mathcal{F}^\Upsilon_0}$.

    Then for arbitrary choice of the countable subset $\Upsilon_0$, the symmetric space
$$
\mathbf{E}(\Omega^\Upsilon,\mathcal{F}^\Upsilon,\mu^\Upsilon) \cap \mathbf{L}_0(\Omega^\Upsilon,\mathcal{F}^\Upsilon_0,\mu^\Upsilon_0)
$$
 is isometrically isomorphic to the same standard symmetric spaces \ $\mathbf{E}(I,\mathcal{B}_m, m)$. \  However, it need not be necessarily isomorphic to the original symmetric spaces $\mathbf{E}(\Omega^\Upsilon,\mathcal{F}^\Upsilon,\mu^\Upsilon)$ if
 $\Upsilon$ is uncountable.

\section{Some consequences and applications}
\label{s Some consequences and applications}

\subsection{The equality $\xi_{f,\mu} = f \circ \phi$}
 \label{ss The equality xifm = f circ phi}

As a straightforward consequence of  Theorem \ref{th isomorphism Phi Separable case} we are able to obtain the following useful result.

\begin{theorem}
 \label{th xifm = f circ phi}
  Let $f \in \mathbf{L}_0^\xi(\Omega,\mathcal{F}_\mu, \mu) $ be a measurable nonnegative function such that $\xi_{f,\mu} (\infty) =0$.
  Then there exists a measure preserving mapping $\phi \colon (\Omega,\mathcal{F}_\mu, \mu) \to (I,\mathcal{B}_m, m)$ such that $\xi_{f,\mu}  =f \circ \phi$.
\end{theorem}
\begin{proof}
By Theorem \ref{th isomorphism Phi Separable case} we may assume without loss of generality that the considered
 measure space is standard.

     For every $f \in \mathbf{L}_0^\xi(\Omega,\mathcal{F}_\mu, \mu) $ denote by $\zeta_f$ the measurable  partition of $(I,\mathcal{B}_m, m)$ which consists of elements of the form
$$
  \zeta_f (x) =
      \left\{
       \begin{array}{ll}
           f^{-1}(f(x))  & , \;\mbox{ if}\; m\,(\{f = x\}) =0  \\
           x             & , \;\mbox{ if}\; m\,(\{f = x\}) > 0
       \end{array}
     \right.
$$
    Let $\pi_{\zeta_f} $ be the natural projection $(I,\mathcal{B}_m, m)$  on the factor-space $(I/\zeta_f, \mathcal{B}_m/\zeta_f,m/\zeta_f)$.
    The assumption  $\xi_{f,\mu}(\infty) =0$ provides that the factor-measure $m/\zeta_f$ is $\sigma$-finite, and
 $(I/\zeta_f, \mathcal{B}_m/\zeta_f,m/\zeta_f)$ is a Lebesgue space since the function $f$ and hence the partition $\zeta_f$ are measurable.

    The function $\hat{f}$ defined on $I/\zeta_f$ by the equality $f = \hat{f} \circ \pi_{\zeta_f}$ is equimeasurable to $f$ since $\pi_{\zeta_f}$ is measure
 preserving.
    By constructing, the partition $\zeta_{\hat{f}}$ consists of separate points of $I/\zeta_f$ and the partition $\zeta_{\xi_{f,\mu}}$ has just the
 same property on $I$.

    Thus there is a measure preserving isomorphism $\tau \colon (I/\zeta_f, \mathcal{B}_m/\zeta_f,m/\zeta_f) \to (I,\mathcal{B}_m,m)$ such that
 $\hat{f} = \xi_{f,\mu} \circ \tau$.
    The mapping $\phi = \tau \circ \pi_{\zeta_f} $ is desired.
\end{proof}

Theorem \ref{th xifm = f circ phi} seems to be known for long time.
    An explicit proof was given in \cite{BeSh}, sec. 2.7 and earlier in \cite{Ry}, in the case $\Omega=\mathbf{R}$.
    If $(\Omega,\mathcal{F}_\mu, \mu) $ is a Lebesgue space, the result follows from Rokhlin's classification of measurable functions \cite{Ro$_2$}. The above proof is new.

\subsection{ Minimality. Properties (A) and  (C)}
 \label{ss Minimality. Properties (A) and (C)}

  A symmetric space $\mathbf{E }= \mathbf{E}(\Omega, \mathcal{F}, \mu)$ is called {\it minimal} $\,$ if the set $\mathbf{F}_1$ of all simple (finite-valued) integrable
 functions on $\Omega$ is dense in $\mathbf{E}$, i.e. the closure $cl_\mathbf{E}(\mathbf{F}_1)$ coincides with $\mathbf{E}$.
     The set $\mathbf{F}_1$ consists of all simple functions $f\colon \Omega \to [0, \infty)$ such that $0 \leq \mu(\{|f|=a\}) < \infty $
 for all $a>0$.

     For every $0<p<\infty$, the set $\mathbf{F}_1$ is dense in $\mathbf{L}_p\cap \mathbf{L}_\infty $ in the norm $\|\cdot\|_{\mathbf{L}_p\cap \mathbf{L}_\infty}$.

\begin{itemize}
           \item[$\bullet$]
    Let $\mathbf{E}$ be a symmetric quasi-Banach  space with the  modulus of concavity $c_\mathbf{E} >1$  and let $0 < p< 1$ such that
  $C=2^{1/p -1} > c_\mathbf{E}$.
     Then $ \mathbf{L}_p\cap \mathbf{L}_\infty \subseteq \mathbf{E}$ and  $\mathbf{E}$ is minimal iff $cl_\mathbf{E} (\mathbf{L}_p\cap \mathbf{L}_\infty) =\mathbf{E}$.

     \item[$\bullet$]
     Let $\mathbf{E}$ is a symmetric Banach space. Then $\mathbf{L}_1\cap \mathbf{L}_\infty \subseteq \mathbf{E}$ and $\mathbf{E}$ is minimal iff $cl_\mathbf{E}(\mathbf{L}_1\cap \mathbf{L}_\infty ) =\mathbf{E}$.
\end{itemize}

In general case, the closure $\mathbf{E}^0 := cl_\mathbf{E} (\mathbf{F}_1) \subseteq \mathbf{E}$ is a symmetric space, which  is referred to as
 the {\it minimal part} of $ \mathbf{E}$.
     So that $\mathbf{E}$ is a minimal symmetric spaces iff $\mathbf{E}^0 = \mathbf{E}$.

     It is clear that two equimeasurable symmetric spaces are or are not minimal, simultaneously.

     A symmetric space $\mathbf{E}$ is said to have property (A)  (order continuous norm) if
\begin{itemize}
         \item[(A)]
    $\;$ If $ 0  \leq f_n \in \mathbf{E}$ and $f_n \downarrow 0$ \ then \ $\|f_n\|_{\mathbf{E}} \downarrow 0$.
     \end{itemize}

    An element $f \in \mathbf{E}$ has an {\it absolutely continuous norm}, if for every $\varepsilon >0$ there exists $\delta > 0$ such that
 $\|\mathbf{1}_A f\| < \varepsilon$ for all $A \in \mathcal{F}_\mu$ with $\mu (A) < \delta$.

\begin{theorem}
 \label{th property (A)}
      Let $\mathbf{E}=\mathbf{E}(\Omega,\mathcal{F}_\mu,\mu)$ be a symmetric  space on $(\Omega,\mathcal{F}_\mu,\mu)$.
     Then the following conditions are equivalent:
 \begin{itemize}
              \item[1.] \  $\mathbf{E}$ has property (A).
  \item[2.] \  Every $f \in \mathbf{E}$ has an absolutely continuous norm.
  \item[3.] \ $\mathbf{E}$ is minimal and $ \varphi_{\mathbf{E}} (0+) = 0$.
 \end{itemize}
\end{theorem}

The result is well-known in the case when $\mathbf{E}$ is a symmetric Banach space.
 See, \cite{KanAk}, X.3, Theorem 3, or  \cite{LTz}, Proposition 1.a.8, Theorem 1.b.16. The statement remains valid if we replace norms by $p$-norms.

     A symmetric space $\mathbf{E}$ is said to have property (C) ({\it order semi-continuous norm}) if
\begin{itemize}
  \item[(C)]
 $\;$  $ 0  \leq f_n \uparrow f \in \mathbf{E}$ \ implies \  $\sup_n \|f_n\|_\mathbf{E} = \|f\|_\mathbf{E}$.
 \end{itemize}

 Note that a symmetric space $\mathbf{E}(\Omega,\mathcal{F}_\mu,\mu)$ has property (C) iff its standard symmetric space $\mathbf{E}(I,\mathcal{B}_m,m)$ has the property.

   It is obvious, that Property (A) implies property (C) and the minimality.

   Minimal symmetric space need not to have property (A).  Nevertheless,

\begin{theorem}
 \label{th minimality implies (C)}
    Every minimal symmetric space $\mathbf{E}$ has property (C).
\end{theorem}
The Theorem was proved in \cite{MuRu}, sec.2, Theorem 2.2 in the case,  when $\mathbf{E}$  is a  standard symmetric  Banach spaces.

The proof is literally carried over to the case of quasi-Banach symmetric spaces $\mathbf{E}$ if we      replace the norm $\|\cdot\|_\mathbf{E}$ by an equivalent $p$-norm $\|\cdot\|$ with $0< p \leq  1$.

     Since both minimality and (C) properties are invariant for equimeasurable symmetric spaces, Theorem \ref{th minimality implies (C)}
 holds true for symmetric spaces on general (not necessarily standard) measure spaces.

    Property (A) as well as minimality are invariant of equimeasurable symmetric spaces.

    The minimality condition of $\mathbf{E}(\Omega,\mathcal{F}_\mu,\mu)$ can be formulated in terms of $\mathbf{E}(I,\mathcal{B}_m,m)$ as follows.

    Let $\min(\xi_{f,\mu}, n)$ and $\xi_{f,\mu} \cdot \mathbf{1}_{[0,n]}$ be upper and right $n$-cutoff functions of $\xi_{f,\mu}$,
 where $f\in \mathbf{E}(\Omega,\mathcal{F}_\mu,\mu)$ and $\xi_{f,\mu} \in \mathbf{E}(I,\mathcal{B}_m,m)$.
\begin{itemize}
 \item[$\bullet$]
    A symmetric space $\mathbf{E}(\Omega,\mathcal{F}_\mu,\mu)$ is minimal iff
$$
  \|\xi_{f,\mu} - \min(\xi_{f,\mu}, n)\|_{\mathbf{E}(I,\mathcal{B}_m,m)} \to 0
  \;\mbox{and}\;  \|\xi_{f,\mu} - \xi_{f,\mu}\cdot \mathbf{1}_{[0,n]}\|_{\mathbf{E}(I,\mathcal{B}_m,m)} \to 0 \;\mbox{as}\; n\to \infty \;.
$$
\end{itemize}

    This simple remark combined with Theorems \ref{th main theorem 1} and \ref{th main theorem 2} yields the following

\begin{theorem}
 \label{th separability and (A)}
     Let $\mathbf{E}(\Omega,\mathcal{F}_\mu,\mu)$ be a symmetric  space on $(\Omega,\mathcal{F}_\mu,\mu)$ and $\mathbf{E}(I, \mathcal{B}_m, m)$ be its standard symmetric Banach space on   the corresponding standard measure space $(I, \mathcal{B}_m, m)$.
    Then
\begin{enumerate}

  \item[1.]
   The following conditions are equivalent:
 \begin{itemize}
  \item[$\bullet$]
   $\mathbf{E}(\Omega,\mathcal{F}_\mu,\mu)$ has property (A).
  \item[$\bullet$]
   $\mathbf{E}(I,\mathcal{B}_m,m)$ has property (A).
  \item[$\bullet$]
   $\mathbf{E}(I,\mathcal{B}_m,m)$ is separable.
\end{itemize}
  \item[2.]
   The following conditions are equivalent:
  \begin{itemize}
    \item[$\bullet$]
     $\; \mathbf{E}(\Omega,\mathcal{F}_\mu,\mu)$ has property (A) and the measure space $(\Omega,\mathcal{F}_\mu,\mu)$ is separable.
    \item[$\bullet$]
     $\; \mathbf{E}(I,\mathcal{B}_m,m)$ has property (A) and the measure space $(\Omega,\mathcal{F}_\mu,\mu)$ is separable.
    \item[$\bullet$]
     $\; \mathbf{E}(I,\mathcal{B}_m,m)$ is separable and the measure space $(\Omega,\mathcal{F}_\mu,\mu)$ is separable.
   \end{itemize}
 \end{enumerate}
 \end{theorem}
We observe that separability (in contrast to minimality and property (A)) is not invariant of equimeasurable symmetric spaces

\subsection{Associate symmetric spaces}
 \label{ss Associate symmetric spaces}$\;$

  Let $\mathbf{E}=\mathbf{E}(\Omega,\mathcal{F}_\mu,\mu)$ be a symmetric space on $(\Omega,\mathcal{F}_\mu,\mu)$.

  The {\it associate (or K\"{o}the dual)} space $(\mathbf{E}(\Omega,\mathcal{F}_\mu,\mu))^1$ of a symmetric space $\mathbf{E}(\Omega,\mathcal{F}_\mu,\mu)$ is defined as
$$
      (\mathbf{E}(\Omega,\mathcal{F}_\mu,\mu))^1 :=  \left\{ g \in \mathbf{L}_0(\Omega,\mathcal{F}_\mu,\mu) \colon \|g\|_{(\mathbf{E}(\Omega,\mathcal{F}_\mu,\mu))^1} < \infty  \right\} \;,
$$
 where
$$
   \|g\|_{(\mathbf{E}(\Omega,\mathcal{F}_\mu,\mu))^1} := \sup \left \{ \int_\Omega f g \,d\mu, \  \|f\|_{\mathbf{E}(\Omega,\mathcal{F}_\mu,\mu)} \leq 1 \right\}\;.
$$

\begin{proposition}
      Let $\mathbf{E}(\Omega,\mathcal{F}_\mu,\mu)$ be a symmetric Banach space on $(\Omega,\mathcal{F}_\mu,\mu)$ \ and \ $\mathbf{E}(I, \mathcal{B}_m, m)$ be its standard symmetric Banach space on
  the corresponding standard measure space $(I, \mathcal{B}_m, m)$.
 \begin{enumerate}
  \item[1.]
      The spaces $(\mathbf{E}(\Omega,\mathcal{F}_\mu,\mu)^1$ and $(\mathbf{E}(I, \mathcal{B}_m, m))^1$ are symmetric Banach spaces on $(\Omega,\mathcal{F}_\mu,\mu)$ and
   $(I, \mathcal{B}_m, m)$, respectively.
  \item[2.]
       The spaces $(\mathbf{E}(\Omega,\mathcal{F}_\mu,\mu)^1$ and $(\mathbf{E}(I, \mathcal{B}_m, m))^1$ are strictly equimeasurable and
 $$
      \|g\|_{(\mathbf{E}(\Omega,\mathcal{F}_\mu,\mu))^1} =  \|\xi_{g,\mu}\|_{(\mathbf{E}(I,\mathcal{B}_m,m))^1} =
      \sup \left \{ \int_I \xi_{f,\mu} \,\xi_{g,\mu} \,d\mu, \ \|f\|_{\mathbf{E}(I, \mathcal{B}_m, m)} \leq 1 \right\}\;.
 $$
  \end{enumerate}
 \end{proposition}

\begin{proof}
    Both statements 1 and 2 are well-known in the case, when the space $\mathbf{E}(\Omega,\mathcal{F}_\mu,\mu)$ is standard, i.e.
 $\mathbf{E}(\Omega,\mathcal{F}_\mu,\mu) = \mathbf{E}(I, \mathcal{B}_m, m)$, see, for example, \cite{KrPiSe}, ch. II, sec. 4 or \cite{RuGrMuPa}, ch.7.

    For general ($\sigma$-finite, non-atomic) measure space $(\Omega,\mathcal{F}_\mu,\mu)$ and every fixed $g \in \mathbf{E}(\Omega,\mathcal{F}_\mu,\mu)$,
 we can find a countable subalgebra $\mathcal{A} \in \mathfrak{A}$ as in Section \ref{ss The proof of theorems 1} such that
 $g \in \mathbf{L}_0(\Omega,\mathcal{F}_\mathcal{A},\mu_\mathcal{A})$.
    Here $ \mathcal{F}_\mathcal{A} = \mathcal{F}(\mathcal{A})_\mu $ and the restricted measure $\mu_\mathcal{A} =  \mu|_{\mathcal{F}(\mathcal{A})_\mu}$ is $\sigma$-finite, while the measure space
 $(\Omega,\mathcal{F}_\mathcal{A},\mu_\mathcal{A})$ is separable.

    The $\sigma$-finiteness of $\mu_\mathcal{A}$ provides that we may use the conditional expectation

$$
   \mathbb{E}^{\mathcal{F}_\mathcal{A}}_\mu \colon (\mathbf{L}_1+\mathbf{L}_\infty)(\Omega,\mathcal{F}_\mu,\mu) \in f \to \mathbb{E}^{\mathcal{F}_\mathcal{A}}_\mu [f] \in  (\mathbf{L}_1+\mathbf{L}_\infty)(\Omega,\mathcal{F}_\mathcal{A},\mu_\mathcal{A})
$$
 even if $\mu(\Omega) = \infty$, see,  \cite{EdSu}, sec. (2.3.7)-(2.3.9).

    Elementary properties of conditional expectations yield
$$
  \|g\|_{(\mathbf{E}(\Omega,\mathcal{F}_\mu,\mu))^1} = \sup_{\|f\|_{\mathbf{E}(\Omega,\mathcal{F}_\mu,\mu)} \leq 1} \int_\Omega f g\,d\mu =
  \sup_{\|f\|_{\mathbf{E}(\Omega,\mathcal{F}_\mu,\mu)} \leq 1}\int_\Omega \mathbb{E}^{\mathcal{F}_\mathcal{A}}_\mu [f] \, g\,d\mu =
$$
$$
  =\sup_{\|h\|_{\mathbf{E}(\Omega,\mathcal{F}_\mathcal{A},\mu_\mathcal{A})} \leq 1} \int_\Omega h g\,d\mu_\mathcal{A} = \|g\|_{(\mathbf{E}(\Omega,\mathcal{F}_\mathcal{A},\mu_\mathcal{A}))^1} \;.
$$
    On the other hand, the measure space $(\Omega,\mathcal{F}_\mathcal{A},\mu_\mathcal{A})$ is separable. Therefore, we can use an algebraic, order, topological, integral preserving isomorphism
$$
 \Phi_\mathcal{A} \colon \mathbf{L}_0(\Omega,\mathcal{F}_\mathcal{A},\mu_\mathcal{A}) \to \mathbf{L}_0(I,\mathcal{B}_m, m) \;,
$$
 provided by  Theorem \ref{th isomorphism Phi Separable case}.
    The isomorphism induces an isometric isomorphism between the spaces $\mathbf{E}(\Omega,\mathcal{F}_\mathcal{A},\mu_\mathcal{A})$ and $ \mathbf{E}(I, \mathcal{B}_m, m)$,
 and hence, determines an isometric isomorphism between the spaces $(\mathbf{E}(\Omega,\mathcal{F}_\mathcal{A},\mu_\mathcal{A}))^1$ and $ (\mathbf{E}(I, \mathcal{B}_m, m))^1$.

    Whence $ (\mathbf{E}(I, \mathcal{B}_m, m))^1$ is the standard space of  $(\mathbf{E}(\Omega,\mathcal{F}_\mathcal{A},\mu_\mathcal{A}))^1$ and of $(\mathbf{E}(\Omega,\mathcal{F}_\mu,\mu))^1$ as well.
    \end{proof}

   We usually prefer  to write $\mathbf{E}^1(\Omega,\mathcal{F}_\mu,\mu)$ and $\mathbf{E}^1(I, \mathcal{B}_m, m)$ instead $(\mathbf{E}(\Omega,\mathcal{F}_\mu,\mu))^1$ and
 $(\mathbf{E}(I, \mathcal{B}_m, m))^1$.

    The associate space $\mathbf{E}^1$ of a symmetric Banach space $\mathbf{E}$ can be identified with the subset $ \{ \upsilon_g \colon g \in \mathbf{E}^1\}$ of
 the dual space $\mathbf{E}^*$, where
$$
  \upsilon_g(f) := \int fg \,d\mu \,,\, f \in \mathbf{E} \;.
$$
    By definition $\upsilon_g \in \mathbf{E}^*$ and $\|\upsilon_g\|_{\mathbf{E}^*} = \|g\|_{\mathbf{E}^1}$ for every $g \in \mathbf{E}^1$.

    The natural embedding $\upsilon : \mathbf{E}^1 \ni g  \rightarrow \upsilon_g \in \mathbf{E}^*$ is an isometric isomorphism from $\mathbf{E}^1$ onto a closed subspace
 $\upsilon(\mathbf{E}^1)=\{ \upsilon_g, g \in \mathbf{E}^1 \}$ of $\mathbf{E}^*$.

    It should be noted that $\upsilon(\mathbf{E}^1)$ may be a proper subset of $\mathbf{E}^*$ \ in general.
    For example, $\upsilon(\mathbf{E}^1) \subset \mathbf{E}^*$ if $\mathbf{E}=\mathbf{L}_\infty$.

    Theorem \ref{th property (A)} can be completed now by one more equivalent condition.

\begin{theorem}
 \label{th E^1 = E^*}
    Let $\mathbf{E}=\mathbf{E}(\Omega,\mathcal{F}_\mu,\mu)$ be a symmetric Banach space.

    Then  $\upsilon(\mathbf{E}^1) = \{\upsilon_g, \ g \in \mathbf{E}^1 \} = \mathbf{E}^*$ \ iff \ $\mathbf{E}$ has property (A).
\end{theorem}

\begin{remark}

Above definitions and results formally make sense for general symmetric quasi-Banach spaces $\mathbf{E}(\Omega,\mathcal{F}_\mu,\mu)$.
    However, since we deal only with non-atomic measure spaces $ (\Omega,\mathcal{F}_\mu,\mu)$, the dual space $\mathbf{E}^*$ and hence, the associate
 space $\mathbf{E}^1$ may be trivial as soon as $\mathbf{E}$ is not normable.

    For, example, let $\mathbf{E}=\mathbf{L}_p, \ 0 < p \leq \infty$.
    For $1 \leq p \leq \infty$, we have $(\mathbf{L}_p)^1 = \mathbf{L}_q$ with $1/p + 1/q =1$.

    The reader can use \cite{Kal}, to find
         useful results and references about dual spaces $\mathbf{E}^*$ of non-normable
 quasi-Banach spaces $\mathbf{E}$.
 \end{remark}

Returning to symmetric Banach spaces we consider the associate spaces $\mathbf{E}^{11} = (\mathbf{E}^1)^1$, \ $\mathbf{E}^{111} = (E^{11})^1$ and  etc.

It follows directly from the definitions that $\mathbf{E} \subseteq \mathbf{E}^{11}$ and  $\mathbf{E}^1 = \mathbf{E}^{111}$, although the inclusion
 $\mathbf{E} \subseteq \mathbf{E}^{11}$ may be strict.
     Moreover,
$$
\|f\|_{\mathbf{E}^{11}} \leq \|f\|_\mathbf{E }\;,\; f \in \mathbf{E} \;,
$$
 i.e. the natural embedding $\mathbf{E} \rightarrow \mathbf{E}^{11}$ is a contraction.

Condition (C) just guarantees that the mapping is an isometry.
    Namely,

\begin{theorem}
 \label{th property (C)}
 Let $\mathbf{E}=\mathbf{E}(\Omega,\mathcal{F}_\mu,\mu)$ be a symmetric Banach space.
   Then the following conditions are equivalent:
 \begin{enumerate}
  \item
   $\;\mathbf{E}$ has property (C).
  \item
   $\; \|f\|_\mathbf{E} = \sup\,\{\upsilon_g(f): \|g\|_{\mathbf{E}^1} \leq 1 \}$.
  \item
   $\; \|f\|_{\mathbf{E}^{11}} =  \|f\|_\mathbf{E} \,,\, f\in \mathbf{E} $.
 \end{enumerate}
\end{theorem}

The theorem is usually referred to as Nakano-Amemiya-Mori theorem, see,
\cite{MoAmNa} or \cite{KanAk}, Theorem X.4.7. A subspace $\mathbf{G}\subseteq \mathbf{E}^*$ is said to be {\it norming} if  $\; \|f\|_\mathbf{E} = \sup\,\{g \in \mathbf{G}, \ \|g\|_{\mathbf{E}^1} \leq 1\}$,
 see, \cite{LTz}, sec. 1.b.
     Condition 2 in the above theorem means that the subspace $\upsilon(\mathbf{E}^1)$  of $\mathbf{E}^*$ is norming.

    Let, for example, $\mathbf{E}=\mathbf{L}_\infty$ with a new norm
$$
  \|f\|_\mathbf{E} = \|f\|_{\mathbf{L}_\infty} + \xi_{f,\mu}(\infty), \  f \in \mathbf{E} \  \;\mbox{where}\; \ \xi_{f,\mu}(\infty) = \lim_{x \to \infty} \xi_{f,\mu}(x) \;.
$$
    Then $(\mathbf{E},\|\cdot\|_\mathbf{E})$ is a symmetric Banach space which fails to have Property (C).

   In this example the natural embedding $\mathbf{E} \to \|\cdot\|_{\mathbf{L}_\infty}$ is not  isometric.  However, the norms $\|\cdot\|_{\mathbf{E}}$ and $\|\cdot\|_{\mathbf{L}_\infty}$ are equivalent, since for every $f \in \mathbf{E} = \mathbf{L}_\infty$,
$$
     \xi_{f,\mu}(\infty) \leq \xi_{f,\mu}(0) =  \|f\|_{\mathbf{L}_\infty}, \ f \in \mathbf{E }= \mathbf{L}_\infty
     \Longrightarrow   \|\cdot\|_{\mathbf{L}_\infty} \leq \|\cdot\|_\mathbf{E} \leq  2\|\cdot\|_{\mathbf{L}_\infty} \;.
$$
 Therefore, the identity $\mathbf{E} \to \mathbf{E}^{11}$ is open.

 In general case the embedding $\mathbf{E} \to \mathbf{E}^{11}$ is open iff $\mathbf{E}$ is closed in $\mathbf{E}^{11}$ by Open Mapping Theorem.

    However, there exist symmetric Banach space for which $\mathbf{E}$ is not closed in $\mathbf{E}^{11}$  and the identity mapping $\mathbf{E}\rightarrow \mathbf{E}^{11}$ is not open.

\medskip

\subsection{ Maximality. Properties (B) and (F)}
 \label{ss Maximality. Properties (B) and (F)} $\;$

\medskip

 A symmetric space $\mathbf{E}=\mathbf{E}(\Omega,\mathcal{F}_\mu,\mu)$ is said to have property (B) ({\it monotonically complete norm}) if
 \begin{itemize}
  \item[(B)]
   $\;$ If $ 0  \leq f_n \uparrow \,,\, f_n \in \mathbf{E} \,,\, \sup_n \|f_n\|_\mathbf{E} < \infty$ \ then \
    $f_n \uparrow f$ for some $f \in \mathbf{E}$.
 \end{itemize}

   Note that this property (as well as property (A) and property(C)) is invariant for equimeasurable symmetric spaces.

    In the case of symmetric Banach spaces we may use the embedding $\mathbf{E} \subseteq \mathbf{E}^{11}$.

\begin{theorem}
 \label{th property (B)}
 Let $\mathbf{E}=\mathbf{E}(\Omega,\mathcal{F}_\mu,\mu)$ be a symmetric Banach space.
   Then the following conditions are equivalent:
    \begin{enumerate}
  \item
   $\;\mathbf{E}= \mathbf{E}^{11}$ as sets, i.e. $\mathbf{E}$ is maximal.
  \item
  $\; \mathbf{E} = \mathbf{G}^1$ for some symmetric space $\mathbf{G}=\mathbf{G}(\Omega,\mathcal{F}_\mu,\mu)$.
  \item
  $\;\mathbf{E}$ has property (B).
 \end{enumerate}
\end{theorem}

    Combining Theorems \ref{th property (C)} and \ref{th property    (B)} we have

\begin{theorem}
 \label{th property (BC)}
 Let $\mathbf{E}=\mathbf{E}(\Omega,\mathcal{F}_\mu,\mu)$ be a symmetric space.
   Then the following conditions are equivalent:
   \begin{enumerate}
  \item \
   $\mathbf{E} = \mathbf{E}^{11}$ and $\|\cdot\|_\mathbf{E} = \|\cdot\|_{\mathbf{E}^{11}} $.
  \item \
  $\mathbf{E} = \mathbf{G}^1$ and $\|\cdot\|_\mathbf{E} = \|\cdot\|_{\mathbf{G}^1} $ for some symmetric space  $\mathbf{G}=\mathbf{G}(\Omega,\mathcal{F}_\mu,\mu)$.
  \item \
  $\mathbf{E}$ has both property (B) and (C).
  \item \
    If $\kappa \colon\mathbf{ E} \to \mathbf{E}^{**}$ is the canonic embedding, then there exists projection $\pi \colon \mathbf{E}^{**} \to \kappa(\mathbf{E})$ with $\|\pi\|=1$.
   \item \
     Every centered sequence of closed balls has nonempty intersection in $\mathbf{E}$.
 \end{enumerate}
\end{theorem}

See, \cite{KanAk}, Theorem X.4, where this and previous results are treated for general Banach  $K$-spaces.

Combination of properties (B) and (C) (used in the above theorem) means
 \begin{itemize}
  \item[(BC)]
    $\;$ If $ 0  \leq f_n \uparrow  \,,\, f_n \in \mathbf{E} $ and $ \sup_n \|f_n\|_\mathbf{E} < \infty $ \ then  \ $f_n \uparrow f$ for some
    $f \in \mathbf{E}$ and $\sup_n \|f_n\|_{\mathbf{E}} = \|f\|_{\mathbf{E}}$.
\end{itemize}
    It makes sense for general symmetric quasi-Banach spaces and can be reformulated as the following Fatou's property (F),
 see. \cite{LTz}, sec. 1.b.
 \begin{itemize}
  \item[(F)]
    $\;$ If $ \{f_n\} \subset \mathbf{E} \,,\, f_n\stackrel{(a.e)}{\rightarrow}  f $ and $\sup_n \|f_n\|_\mathbf{E} < \infty $ then
    $f \in E$ and $\|f\|_\mathbf{E} \leq \liminf_n \|f_n\|_\mathbf{E} $.
\end{itemize}

   \medskip

   In the case $\mathbf{E}=\mathbf{L}_1$ property (F) is just the assertion of Fatou's theorem.

\end{document}